\documentclass[reqno,12pt]{amsart}

\usepackage{ifpdf}
\ifpdf
    \usepackage{thumbpdf}
    \usepackage[pdftex]{graphicx}
    \usepackage[pdftex]{color} 
    \usepackage[pdftex,colorlinks]{hyperref}
\else
    \usepackage[dvips]{graphicx}
    \usepackage[dvips]{color} 
    \usepackage[dvips,colorlinks]{hyperref}
\fi

\usepackage{amsmath}
\usepackage{amsthm}
\usepackage{amsfonts}
\usepackage{amssymb}

\usepackage{verbatim}
\usepackage{subfigure}

\newtheorem{theorem}{Theorem}[section]

\newtheorem{definition}[theorem]{Definition}
\theoremstyle{definition}

\newcounter{mnote}
\numberwithin{equation}{section}  
\begin{document}

\title[Bifurcation in the Einstein Constraints]
      {Numerical Bifurcation Analysis of Conformal \\
       Formulations of the Einstein Constraints}

\author[M. Holst]{M. Holst}
\email{mholst@math.ucsd.edu}

\author[V. Kungurtsev]{V. Kungurtsev}
\email{vkungurt@math.ucsd.edu}

\address{Department of Mathematics\\
         University of California San Diego\\ 
         La Jolla CA 92093}

\thanks{MH was supported in part by NSF Awards~0715146 and 0915220,
and by DOD/DTRA Award HDTRA-09-1-0036.
VK was supported in part by NSF Awards~0715146 and 0915220.}

\date{\today}


\begin{abstract}
The Einstein constraint equations have been the subject of study for more
than fifty years.
The introduction of the conformal method in the 1970's as a parameterization
of initial data for the Einstein equations led to increased interest in the 
development of a complete solution theory for the constraints, with the 
theory for constant mean curvature (CMC) spatial slices and closed manifolds
completely developed by 1995.
The first general non-CMC existence result was establish by Holst 
et al.\ in 2008, with extensions to rough data by Holst et al.\ in 2009, 
and to vacuum spacetimes by Maxwell in 2009.
The non-CMC theory remains mostly open; moreover, recent work of Maxwell 
on specific symmetry models sheds light on fundamental non-uniqueness 
problems with the conformal method as a parameterization in non-CMC settings.
In parallel with these mathematical developments, computational physicists 
have uncovered surprising behavior in numerical solutions to the extended 
conformal thin sandwich formulation of the Einstein constraints.
In particular, numerical evidence suggests the existence of multiple 
solutions with a quadratic fold, and a recent analysis of a simplified
model supports this conclusion.
In this article, we examine this apparent bifurcation phenomena in a
methodical way, using modern techniques in bifurcation theory and in
numerical homotopy methods.
We first review the evidence for the presence of bifurcation in the 
Hamiltonian constraint in the time-symmetric case.
We give a brief introduction to the mathematical framework for
analyzing bifurcation phenomena, and then develop the main ideas
behind the construction of numerical homotopy, or path-following, methods
in the analysis of bifurcation phenomena.
We then apply the continuation software package AUTO to this problem,
and verify the presence of the fold with homotopy-based numerical methods.
We discuss these results and their physical significance, which lead to
some interesting remaining questions to investigate further. 
\end{abstract}

\maketitle

\vspace*{-1.0cm}
{\footnotesize
\tableofcontents
}
\vspace*{-2.0cm}

\clearpage

\section{Introduction}\label{S:intro}

Einstein's gravitational field equations for relating the space curvature at a time slice to the stress-energy can be split into a set of evolution and constraint equations.
The four constraint equations, known as the (scalar) Hamiltonian constraint and the (3-vector) momentum constraint, constrain the induced spatial metric $g_{ij}$ and extrinsic curvature $K_{ij}$.
The Einstein constraint equations have been the subject of study for more
than fifty years (cf.~\cite{ADM62}).
The introduction of the conformal method in the 1970's as a way of
parameterizing initial data to the Einstein equations led to increased
interest in the development of a complete solution theory for the constraints,
with the theory for constant mean curvature (CMC) spatial slices and closed 
manifolds completely developed by 1995.
The CMC theory on closed manifolds is particular satisfying, with nearly
all physically interesting cases exhibiting both existence and uniqueness
of solutions.

However, other than the near-CMC result of Isenberg and Moncrief 
in 1996~\cite{IsMo96},
the theory for non-CMC solutions remained completely open until the first
far-from-CMC existence result was establish by 
Holst et al.~\cite{HNT07a} in 2008, with
extensions to rough data by Holst et al.\ in 2009~\cite{HNT07b}, 
and to vacuum spacetimes by Maxwell in 2009~\cite{dM09}.
However, the non-CMC theory remains mostly open, and what is known is
much less satisfying than the CMC case; the new non-CMC results of
Holst et al.\ and Maxwell are based on new types of topological fixed
point arguments, and while they establish existence, these arguments do
not give uniqueness.
Moreover, more recent work of Maxwell on specific symmetry models show in
fact that uniqueness is lost, and also sheds light on some fundamental
problems with the conformal method itself as a parameterization of initial
data on manifolds that are not very close to CMC.

Several decompositions of the equations have been formulated in addition
to the conformal method, although the solution theory for only the conformal
method is well-developed (cf.~\cite{BaIs03}).
In particular, in the extended conformal thin sandwich (XCTS) decomposition 
of the constraints, a conformal factor $\psi$, the lapse $N$, and the shift 
$\beta_i$, are solved for in five coupled elliptic equations with supplied 
background data. 
In parallel with the mathematical developments in the theory for the
conformal method, computational physicists have uncovered surprising
behavior in numerical solutions to the XCTS formulation.
In particular, numerical evidence suggests the existence of multiple 
solutions; Pfeiffer and York in~\cite{PY} numerically construct two solutions 
for a specific choice of free data.

In addition, there have been more theoretical results suggesting non-uniqueness. Under additional assumptions of conformal flatness and time-symmetry, the 
Hamiltonian constraint becomes a decoupled equation for the conformal factor.
Baumgarte, Murchadha and Pfeiffer~\cite{BMP} solve this form of the 
Hamiltonian constraints using Sobolev functions, to find that an equation 
parameter has two admissible values consistent with the constraints.
Walsh~\cite{Walsh} performs a Lyapunov-Schmidt analysis (a standard technique
in bifurcation theory) of this form of the Hamiltonian constraint, and
shows that at a critical point with two expected solution branches, the 
solution branches should take the form of a quadratic fold.
In this article, we examine this apparent bifurcation phenomena in a
methodical way, using modern techniques in bifurcation theory and in
numerical homotopy methods.

{\em Outline of the paper.}
The remainder of the paper is structured as follows.
In \S\ref{S:conformal}, give an overview of the constraint equations
in the Einstein system and in particular the XCTS formulation thereof.
In \S\ref{S:theory}, 
we give a brief introduction to the mathematical framework for
analyzing nonlinear operators in general. We build on this in \S\ref{S:bifurtheory} where we expound on the theoretical framework underlying bifurcation analysis.
In \S\ref{S:homotopy}, we develop the main ideas behind the construction of 
numerical homotopy, or path-following, methods in the numerical treatment 
of bifurcation phenomena, and subsequently in \S\ref{S:setup} present the methodology of the numerical techniques we perform for this problem.
In \S\ref{S:numerical},
we apply the continuation software package AUTO to the constraint problem,
and verify the presence of the fold with homotopy-based numerical methods.
We confirm the earlier results, as well as provide a framework for
a more careful exploration of the solution theory for various
parameterizations of the constraint equations.
Analyzing the Hamiltonian constraint for time-symmetric conformally flat 
initial data, as in~\cite{Walsh,BMP}, we demonstrate the existence and 
location of a critical point, evidence that the solution branch at the 
critical point forms a one-dimensional fold, and the form of the solution as continued 
past the critical point. 

\section{Conformal Thin Sandwich Decomposition}\label{S:conformal}

In General Relativity it is common to look at the curvature of spatial hypersurfaces taken at time-slices of space-time.
The Einstein constraint equations are conditions on the induced spatial metric $g_{ij}$ and the second fundamental form $K_{ij}$ for being a spatial slice.
In addition, there are six evolution equations that govern how this geometric data evolves in a full spacetime.
In solving the equations, a variety of decompositions have been proposed.
In the XCTS formulation, proposed by York~\cite{York}, the shift vector $\beta_i$, the lapse $N$ and a conformal factor $\psi$ are solved for given a set of supplied data that includes $(\tilde{g}_{ij}, \tilde{u}_{ij},K,\partial_t K)$, where $\tilde{g}_{ij}$ is the conformally related induced spatial metric, $\tilde{u}_{ij}$ its time derivative, and $K$ the trace of the extrinsic curvature.
The lapse and shift are formed from taking a level set of $t$ and looking at the normal to the hypersurface $n=N^{-1}(\partial_t-\beta^i \partial_i)$ where $i$ indexes over spatial coordinates~\cite{Isenberg}.
The Hamiltonian constraint of the XCTS equations can be written as
\begin{equation}
\tilde{\nabla}^2\psi-\frac{1}{8}\tilde{R}\psi-\frac{1}{12}K^2\psi^5+\frac{1}{8}\psi^{-7}\tilde{A}^{ij}\tilde{A}_{ij}+2\pi\psi^5\rho=0
\end{equation}
Here $\tilde{\nabla}$ represents the covariant derivative, $\tilde{R}$ the trace of the Ricci tensor and $\tilde{A}$ the trace free part of the extrinsic curvature, all associated with the conformally scaled metric $\tilde{g}_{ij}$.
Following~\cite{BMP,Walsh} we assume time-symmetry (so $K_{ij}=0$) and conformally flat initial data. This constraint then reduces to 
\begin{equation}\label{eq:maineq}
\nabla^2 \psi+2\pi\rho\psi^5=0
\end{equation}
Following~\cite{BMP} we let $\rho$ be the constant mass-density of a star.
Without loss of generality, we take $\rho=0$ outside $r=1$ and look for solutions to \eqref{eq:maineq} with boundary conditions
\begin{align}\label{eq:bc}
\frac{\partial\psi}{\partial r}&=0, \quad r=0 \\
\psi&=1, \quad r=1
\end{align}
We can note that the maximum principle does not apply for this equation and hence uniqueness is not guaranteed.
We will see that this equation has two, one, or no solutions, depending on the value of $\rho$.

\section{Nonlinear Operators on Banach Spaces}\label{S:theory}


Let $X$ and $Y$ be Banach spaces,
and
let $X'$ and $Y'$ be their 
respective dual spaces.
Given a (generally nonlinear) map $F \colon X \to Y$,
we are interested in the following general problem:
\begin{equation}
\label{eqn:problem1a}
\mbox{Find}~u \in X ~\mbox{such~that}~ F(u) = 0 \in Y.
\end{equation}
We will give a brief overview of techniques for analyzing 
solutions to~\eqref{eqn:problem1a}, and for characterizing
their behavior with respect to parameters.

We say that the problem ${F(u)=0}$ is {\em well-posed} if there is
(a) existence, (b) uniqueness, and (c) continuous dependence of
the solution on the data of the problem.
Recall that if $F$ is both one-to-one ({\em injective}) 
and onto ({\em surjective}),
it is called a {\em bijection}, 
in which case the inverse mapping $F^{-1}$ exists, and 
we would have both existence and uniqueness of solutions
to the problem ${F(u)=0}$.
Recall that $F\colon X \to Y$ is a continuous map 
from the normed space $X$ to the normed space $Y$ 
if $\lim_{j \to \infty} u_j = u$ implies that
$\lim_{j \to \infty} F(u_j) = F(u)$, where 
$\{ u_j \}$ is a sequence, $u_j \in X$.
If both $F$ and $F^{-1}$ are continuous, 
then $F$ is called a {\em homeomorphism}.
If both $F$ and $F^{-1}$ are differentiable
(see the below for the definition of differentiation 
of abstract operators in Banach spaces), 
then $F$ is called a {\em diffeomorphism}.
If both $F$ and $F^{-1}$ are $k$-times continuously differentiable, 
then $F$ is called a $C^k$-{\em diffeomorphism}.
A linear map between two vector spaces is a type of {\em homomorphism}
(structure-preserving map);
a linear bijection is called an {\em isomorphism}.

We will need to assemble just a few basic concepts involving general nonlinear 
maps in Banach spaces, to provide the mathematical framework for the 
discussions in the next section.
In particular, the following notion of differentiation of maps on Banach
spaces will be required
(see \cite{StHo2011a,Zeid91a} for more complete discussions).
\begin{definition}
Let $X$ and $Y$ be Banach spaces,
let $F\colon X \to Y$,
and let $D \subset X$ be an open set.
Then the map $F$ is called {\em Fr\'{e}chet-} or {\em {\bf F}-differentiable} 
at $u \in D$ if there exists a bounded linear 
operator $F_u(u)\colon X\to Y$ such that:
$$
\lim_{\| h \|_X \to 0} \frac{1}{\| h \|_X} 
   \| F(u+h) - F(x) - F_u(u)(h)\|_Y = 0.
$$
\end{definition}
The bounded linear operator $F_u(u)$ is called the 
{\em {\bf F}-derivative} of $F$ at $u$.
The {\em {\bf F}-derivative} of $F$ at $u$ can again be shown to be unique.
If the {\bf F}-derivative of $F$ exists at all points $u \in D$, then 
we say that $F$ is {\bf F}-differentiable on $D$.
If in fact $D=X$, then we simply say that $F$ is {\bf F}-differentiable, and
the derivative $F_u(\cdot)$ defines a map from $X$ into the space
of bounded linear maps, $F_u\colon X \to \mathcal{L}(X,Y)$.
In this case, we say that $F \in C^1(X;Y)$.
Many of the properties of the derivative of smooth functions over domains
in $\mathbb{R}^n$ carry over to this abstract setting, including
the chain rule: If $X$, $Y$, and $Z$ are Banach spaces, and
if the maps $F\colon X\to Y$ and $G\colon Y\to Z$ are differentiable,
then the derivative of the composition map $H = G\circ F$ also exists, 
and takes the form
$$
H_u(u) = (G\circ F)_u(u) = G_F(F(u)) \circ F_u(u),
$$
where 
$H_u\colon X\to \mathcal{L}(X,Z)$,
$F_u\colon X\to \mathcal{L}(X,Y)$,
and $G_F\colon Y\to \mathcal{L}(Y,Z)$.
Higher order Fr\'{e}chet (and G\^{a}teaux) derivatives can be defined in the 
obvious way, giving rise to multilinear maps and giving meaning to the 
notation $F \in C^k(X;Y)$.
Note that below we will often encounter functions of two 
variables $F(u,\lambda)$, and will be interested in the derivatives
of such functions with respect to each variable; we will denote these
using the consistent notation $F_u$ and $F_{\lambda}$.
See~\cite{AMR88,StHo2011a} for more complete discussions of 
general maps and differentiation in Banach spaces.

A fundamental concept concerning linear operators 
that we will need for the discussions below is that of a 
{\em Fredholm operator}.
Let $X$ and $Y$ be Banach spaces, and let $A \in \mathcal{L}(X,Y)$,
or in other words, $A$ is a bounded linear operator from $X$ to $Y$.
In the case that $X$ and $Y$ have additional Hilbert space structure, wherein there is an inner-product,
the Riesz Representation Theorem implies that the adjoint operator
$A^* \in \mathcal{L}(Y,X)$, defined as the operator for which $(Ax,y)=(x,A^* y)$, exists uniquely.
There are four fundamental subspaces of $X$ and $Y$ associated with $A$
and $A^*$, namely:
\begin{enumerate}
\item $\mathcal{N}(A):$ null space (or {\em kernel}) of $A$
\item $\mathcal{R}(A):$ range space of $A$
\item $\mathcal{N}(A^*):$ null space (or {\em kernel}) of $A^*$
\item $\mathcal{R}(A^*):$ range space of $A^*$
\end{enumerate}
One can consider the dimension (dim) and co-dimension (codim) 
of each of these four spaces, where co-dimension is taken to be
relative to the larger spaces that they are subspaces of.
The operator $A \in \mathcal{L}(X,Y)$ is called a {\em Fredholm operator}
if and only if:
\begin{enumerate}
\item $\mathrm{dim}(\mathcal{N}(A)) < \infty$
\item $\mathrm{dim}(\mathcal{R}(A)) < \infty$
\end{enumerate}
The difference of these two dimensions, namely
\begin{equation}
\mathrm{ind}(A) = \mathrm{dim}(\mathcal{N}(A)) - \mathrm{dim}(\mathcal{R}(A))
\end{equation}
is called the {\em Fredholm index} of $A$.
A basic result about Fredholm operators is the following.
\begin{theorem}
If $A \in \mathcal{L}(X,Y)$ is a Fredholm operator, then
\begin{enumerate}
\item $\mathcal{R}(A)$ is closed.
\item $A^*$ is Fredholm with index $\mathrm{ind}(A^*) = - \mathrm{ind}(A)$.
\item $K$ is a compact operator, then $A+K$ is Fredholm with
      $\mathrm{ind}(A+K) = \mathrm{ind}(A)$.
\end{enumerate}
\end{theorem}
\begin{proof}
See~\cite{Zeid91a}.
\end{proof}

\section{Bifurcation Theory for Nonlinear Operators Equations}\label{S:bifurtheory}

Bifurcation Theory studies the branching of solutions as governed by parameter(s). Throughout this section we will refer to a "solution curve" which, for $F(u,\lambda)$ plots points in $(||u||,\lambda)$ space at which $(u,\lambda)$ solves $F(u,\lambda)=0$, and $||u||$ denotes the norm of $u$. Specifically, we are interested in the local behavior of the solution curve in a neighborhood of a known solution $(u_0,\lambda_0)$. This is because, in order to explore the solution space to a certain operator equation $F(u,\lambda)$, we often solve $F(u, \lambda_0)$ for $u$, then obtain another $(u_1, \lambda_1)$ from the original solution data we obtain, and similarly continue to subsequent solutions.

Bifurcation theory has its foundation in the implicit function theorem, 
\begin{theorem}[Implicit Function Theorem]
\label{thm:ift}
If it holds that, for an operator $F:X\times \mathbb{R}\rightarrow Y$, if
$F(u,\lambda)$ with $F(u_0,\lambda_0)=0$,
$F$ and $F_u$ are continuous on some region $U\times V$ with $(x_0,\lambda_0)\in U\times V$, and   
$F_u(u_0,\lambda_0)$ is nonsingular with a bounded inverse, then
there is a unique branch of solutions $(u(\lambda),\lambda))$ in for $\lambda\in V$, e.g.  $F(u(\lambda),\lambda)=0$. Furthermore $u(\lambda)$ is continuous with respect to $\lambda$ in $V$
\end{theorem} 
This theorem states that if the operator $F_u$ is nonsingular at a certain point $(x_0,\lambda_0)$ there is a unique solution $u$ for each $\lambda$ close to $\lambda_0$ on either side of $\lambda_0$ and we can plot a one-dimensional curve in $(||u||, \lambda)$ space through $(||u_0||,\lambda_0)$. With a singular $F_u$, however, such a branch is not guaranteed, suggesting the possibility of two or more such $u(\lambda)$ branches or no solutions for some $\lambda$ in every neighborhood around $\lambda_0$.
The exact form of the branching depends on the dimensions of $F_u(x_0,\lambda_0)$, $F_\lambda(x_0,\lambda_0)$ and whether $F_\lambda(x_0,\lambda_0)\in \mathcal{R}(F_u(x_0,\lambda_0))$.

In the case of a "fold", wherein there is still a one-dimensional path through $(u_0,\lambda_0)$ but the path exists solely on one side of $\lambda_0$ for $\lambda$, we have 
\begin{enumerate}
\item $\mathrm{dim}(\mathcal{N}(F_u(u_0,\lambda_0)))=1$
\item $\mathrm{dim}(\mathcal{N}(F_u(u_0,\lambda_0)^*))=1$, where $F_u(u_0,\lambda_0)^*$ is the adjoint operator of $F_u(u_0,\lambda_0)$ 
\item $F_\lambda(u_0,\lambda_0) \notin \mathcal{R}(F_u(u_0,\lambda_0)$.
\end{enumerate}
It can be shown, under these circumstances, near $(u_0,\lambda_0)$ the continuation of the solution will be of the form
\begin{align}
u(\epsilon) &= u_0 +\epsilon \phi + C_1 \epsilon^2 +O(\epsilon^3) \\
\lambda(\epsilon)&=\lambda_0+C_2 \epsilon^2 + O(\epsilon^3)
\end{align}
where $\phi$ is a basis for the null-space of $F_u(u_0,\lambda_0)$ (see~\cite{Bifur}). Note that there is no linear $\lambda(\epsilon)$ term, so the solution curve, at first approximation, stays at a constant value of $\lambda$ and changes in $u$. 
The specific values of $C_1$ and $C_2$ are based on the Hessian of the operator $F$.
With a fold the $C_2$ will be negative, so the solution curve shows $\lambda$ increase up to a certain critical point $\lambda_c$ then decreases. With $C_1<0$ the actual solution $u$ changes at first-order along $\phi$, the change then dampens in magnitude. Hence with $C_2\neq 0$, the resulting solution curve appears locally as a sideways parabola, and hence called a "simple quadratic fold". 
If, on the other hand, $C_2=0$, we have at $(x_0,\lambda_0)$ a "fold of order m" if $\lambda^{(k)}(\epsilon)=0$ for $k<m$~\cite{Keller}.
In the case of a simple singular point, we have either 
\begin{enumerate}
\item $\text{dim} (\mathcal{N}(F_u(u_0,\lambda_0)))=1$, and
\item $F_\lambda(u_0,\lambda_0) \in \mathcal{R}(F_u(u_0,\lambda_0))$,
\end{enumerate}
or we have
\begin{enumerate}
\item $\text{dim} (\mathcal{N}(F_u(u_0,\lambda_0)))=2$, and
\item $F_\lambda(u_0,\lambda_0) \notin \mathcal{R}(F_u(u_0,\lambda_0))$.
\end{enumerate}
In this case, we have a situation of branch-switching, in which there are are two branches of solutions crossing the point $(u_0,\lambda_0)$~\cite{Keller}.
With high-dimensional null-spaces, the situation becomes increasingly complicated, with multiple branching solutions of various forms.
We illustrate the three representative cases as they would appear in $(||u||, \lambda)$ space in Figure~\ref{fig:paths}.

\begin{figure}[t]
\begin{center}
\includegraphics[scale=0.55]{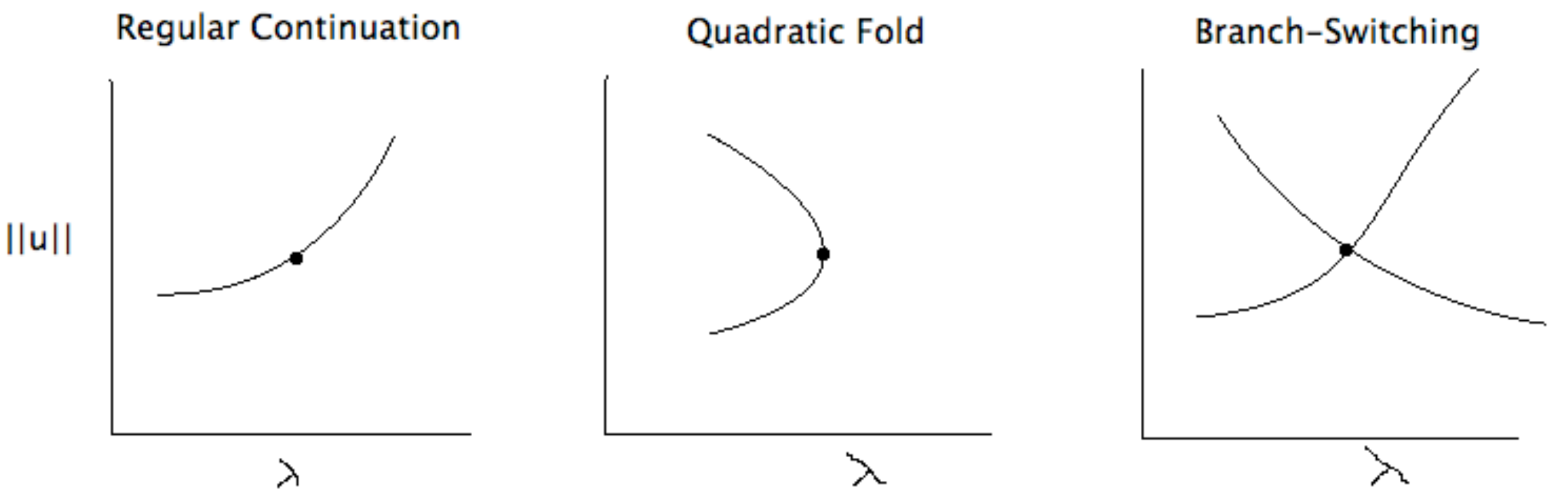}
\end{center}
\caption{\label{fig:paths}Common Locally Bifurcating Solution Paths.}
\end{figure}

We conclude, for completeness, this section with an exposition of the generalized form of the bifurcation analysis method of Lyapunov-Schmidt, as used for this problem by Walsh~\cite{Walsh}. The exposition follows Zeidler~\cite{Zeid91a}.  We assume that
\begin{enumerate}
\item $F_u(u_0,\lambda_0)$ is a Fredholm operator of index $k$
\item $\text{dim}(\mathcal{N}(F_u(u_0,\lambda_0)))=n$
\end{enumerate} 
Then we define projection operators $P:X\rightarrow X$ and $Q:X \rightarrow X$ with $P(X)=\mathcal{N}(F_u(u_0,\lambda_0))$ and $(I-Q)(Y)=\mathcal{R}(F_u(u_0,\lambda_0))$. Now the equation $F(u,\lambda)=0$ is equivalent to the pair
\begin{align}\label{eq:LS}
(I-Q)F(y+z,\lambda)&=0\\
QF(y+z,\lambda)&=0
\end{align}
with $y=(I-P)u$ and $z=Pu$. Now the first equation satisfies the assumptions of the Implicit Function Theorem, and so we can get a unique solution branch $y(z,\lambda)$, substitute the solution in the second equation to obtain the branching equation 
\begin{equation}\label{eq:branch}
QF(y(z,\lambda)+z,\lambda)=0
\end{equation}
then solve for $z(\lambda)$ to get the branch $u=y(z(\lambda),\lambda)+z(\lambda)$
Note that, in practice, one solves \eqref{eq:LS} by expanding the operators in the bases of $\mathcal{N}(F_u(u_0,\lambda_0))$ and $\mathcal{N}(F_u(u_0,\lambda_0)^*)$. Hence, to be constructive, one must already have an estimation of the dimension of the null-space of $F_u(u_0,\lambda_0)$. In particular, Walsh~\cite{Walsh} performs the analysis with the starting assumption that $\text{dim}(\mathcal{N}(F_u(u_0,\lambda_0)))=1$. So while the technique is appropriate once this is known, it is not a standalone method of bifurcation analysis.

\section{Numerical Bifurcation Theory}\label{S:homotopy}

In the procedure of continuation, we seek to find a solution $u$ to a problem $F(u,\lambda)=0$ as we move along $\lambda$. In the case of a nonsingular Jacobian of the operator ($F_u(u_0,\lambda_0)$), it is standard to apply Newton's method on a discretization of the solution space.
So we perform continuation by repeatedly iterating $\lambda=\lambda+\Delta\lambda$ then resolving $F(u,\lambda)=0$ for $u$.
However, the procedure is invalid in the case of a singular Jacobian of the equation operator, necessitating alternatives for traversing a solution as the parameter varies.
Depending on the form of the bifurcation, various procedures exist.
Considering the finite-dimensional case (e.g. for a discretization of $u$), we have that $F(u,\lambda)$ maps $\mathbb{R}^N\times \mathbb{R}$ to $\mathbb{R}^N$. In this case, the rank of $[F_u,F_\lambda]=N$ if either $F_u$ is nonsingular, which is the case above, or $F_\lambda(u_0,\lambda_0) \notin \mathcal{R}(F_u(u_0,\lambda_0))$. In this case $\text{dim}(\mathcal{N}([F_u,F_\lambda]))=1$. Let's say we have a solution $F(u_0,\lambda_0)=0$. In the latter case, we cannot just set $\lambda_1=\lambda_0+\Delta$ (with a constant $\Delta$) and solve for $u_1$, but we can solve for $F(u_1,\lambda_1)=0$ together with one more scalar equation, which we can set to be a constraint on the total magnitude in the change of $(u,\lambda)$.
With pseudo-arclength continuation, at a certain parameter $\lambda_0$ and solution vector $u_0$, and a direction vector $(\dot{u}_0,\dot{\lambda}_0)$ of the solution branch determined thus far, you run Newton's method on the two equations $F(u_1,\lambda_1)=0$ and $(u_1-u_0)^* \dot{u}_0+(\lambda_1-\lambda_0)\dot{\lambda}-\Delta s=0$, where $\Delta s$ is a constant "arclength" term. 
It can be shown that the Newton's method Jacobian matrix is nonsingular if the point is either one at which $F_u$ is nonsingular or a fold~\cite{Bifur}.

If $\text{dim}(\mathcal{N}(F_u))>1$ or  $F_\lambda(u_0,\lambda_0) \in \mathcal{R}(F_u(u_0,\lambda_0))$ then the Jacobian of Newton's method in the pseudo-arclength continuation equations is singular as well, and other procedures must be used to continue the solution. 
 Methods of continuation usually involve constructive techniques, wherein the nullspace vectors for $F_u$ and the solution $F_u \phi_r=F_\lambda$ are explicitly found and coefficients constructed (for a list of relevant algorithms, see~\cite{Keller}). With even higher-dimensional null-spaces, this framework is generalized. 

\section{Setup for Hamiltonian Constraint Bifurcation}\label{S:setup}

We reprint the equations \eqref{eq:maineq}--\eqref{eq:bc} here
\begin{align}
\nabla^2 \psi+2\pi\rho\psi^5 &=0
\label{eq:maineq2}
\\
\frac{\partial\psi}{\partial r} &=0, \quad r=0 
\label{eq:bc2}
\\
\psi &=1, \quad r=1
\end{align}
and perform a reduction for bifurcation analysis following~\cite{Bifur}. 

The goal is to transform this equation into a form wherein its linearization becomes an eigenvalue problem. This can be done by parameterizing an extra boundary condition. The equation can be rewritten in an equivalent form as 
\begin{align}\label{eq:bvp}
\nabla^2 \psi+2\pi\rho\psi^5&=0 \\
\psi(0)&=p \\
\frac{\partial\psi}{\partial r}(0)&=0
\end{align}
Now we seek to solve $F(p,\rho)\equiv \psi(1,p,\rho)-1=0$. Writing $\psi_p(t,p,\rho)=\frac{d\psi}{dp}(t,p,\rho)$, $F_p(p,\rho)=\psi_p(1,p,\rho)$ and similarly $F_\rho(p,\rho)=\psi_\rho(1,p,\rho)$
Now $\psi_p$ satisfies
\begin{align}\label{eq:bvpp}
\nabla^2 \psi_p+10\pi\rho\psi^4\psi_p&=0 \\
\psi_p(0)&=1 \\
\frac{\partial\psi_p}{\partial r}(0)&=0
\end{align}
Which is an eigenvalue problem with a distinct solution. Define $a(p,\lambda)=\psi_p(1,p,\rho)=F_p(p,\rho)$. Similarly $\psi_\rho$ satisfies
\begin{align}\label{eq:bvpr}
\nabla^2 \psi_\rho+2\pi\psi^5+10\pi\rho\psi^4\psi_\rho&=0 \\
\psi_\rho(0)&=0 \\
\frac{\partial\psi_\rho}{\partial r}(0)&=0
\end{align}
and likewise define $b(p,\rho)=\psi_\rho(1,p,\rho)=F_\rho(p,\rho)$.
Now we have $F_{(p,\rho)}=(a,b)$, So $N(F_{(p,\rho)}))$ can be
$\begin{pmatrix}1 \\ 0 \end{pmatrix}$, $\begin{pmatrix}0 \\ 1 \end{pmatrix}$,
both, or neither. Let us discuss each of these cases separately.
If the null-space is empty, then we have a nonsingular Jacobian of the operator, so continuation can proceed with Newton's method as standard. 
If the null-space is 
$\begin{pmatrix}0 \\ 1 \end{pmatrix}$, 
so $b(p,\rho)=0$, then this is a situation where the same solution exists for a differential increase in $\rho$, which is also an uninteresting case. 
If the null-space is 
$\begin{pmatrix}1 \\ 0 \end{pmatrix}$, 
so $a(p,\rho)=0$ and $b(p,\rho)\neq 0$, we now have a case where $N(F_p)=1$ but $F_\rho \notin F_p$, so we have a fold. Pseudo-arclength continuation must be performed, and the form of the continuation gleamed from the coefficients. 
Finally, in the case of a full two-dimensional null-space, we have reached another solution branch. Here, there are two solution branches intersecting, and so we have a choice between three directions in the continuation of the solution.
We now seek to perform this analysis numerically to locate any critical points for \eqref{eq:maineq2} and see one of these scenarios occurs.

\section{Numerical Results}\label{S:numerical}

Recall that a continuation procedure attempts to solve $F(u,\lambda)$ for varying $\lambda$ by a series of successive iterations along discrete steps of $\lambda$. The procedure begins at some $\lambda_0$, finds the solution $F(u_0,\lambda_0)$, finds the nullspace information of $F_u(u_0, \lambda_0)$ and $F_\lambda(u_0,\lambda_0)$, then using this information performs a suitable continuation step to find the next solution $F(u_1,\lambda_1)$.
We used the continuation software AUTO to trace the solutions to \eqref{eq:maineq} with boundary conditions \eqref{eq:bc}. AUTO (\cite{AUTO}) performs numerical continuation on ODEs (notice that due to symmetry the differential equation becomes an ODE with just variable $r$). AUTO calculates the dimension of the null-space of the Jacobian of the differential operator and performs, as required, Newton's method, pseudo-arclength continuation or explicit calculation of the null-space directions if it identifies higher-order bifurcations, as discussed in (\ref{S:homotopy}). It uses the bordering algorithm for the case of pseudo-arclength continuation (\cite{Bifur}) and otherwise more intricate linear algebraic algorithms for higher-dimensional bifurcations.

\begin{figure}[t]
\begin{center}
\includegraphics[scale=0.90]{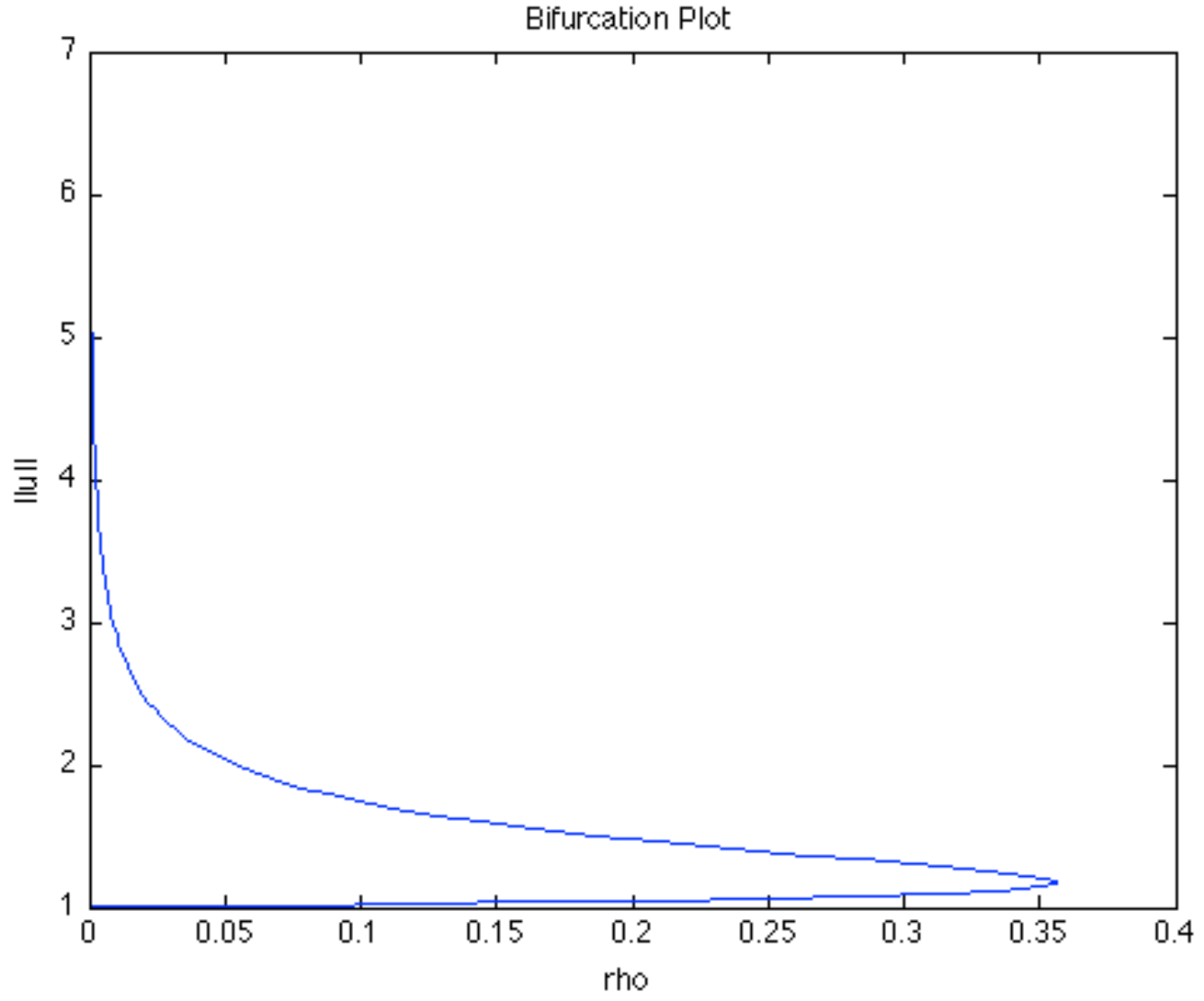}
\end{center}
\caption{Solution curve for \eqref{eq:maineq2}.}
\end{figure}

Continuation reveals a quadratic fold at a value of $\rho$ around $\rho_c \approx 0.35$. For $\rho <\rho_c$ there are two solutions to \eqref{eq:maineq}, at $\rho=\rho_c$, there is exactly one solution, and at $\rho>\rho_c$ there are no solutions. 
AUTO finds that at $\rho=\rho_c$ the nullspace of the discrete Jacobian differential operator has a dimension of one. In addition, it finds that the 2nd derivative information indicates that the continuation has a quadratic form. At all other points along the solution curve, the Jacobian is nonsingular.
Now we let the exponent vary. Writing the equation as $\nabla^2 \psi+2\pi\rho\psi^a=0$, we investigate the continuation for different values of $a$. Knowing that the Laplacian operator on its own has an invertible Jacobian, we expect a fold to appear for some value of $a$. We find that indeed, with $a>1$ a fold appears, the curvature of the solution continuation becoming sharper with increasing $a$. 
In Figure~\ref{fig:bifall} we see the bifurcation diagram for $a=1,1.25,5,10$.
\begin{figure}[tbh]
\begin{center}
\includegraphics[scale=0.68]{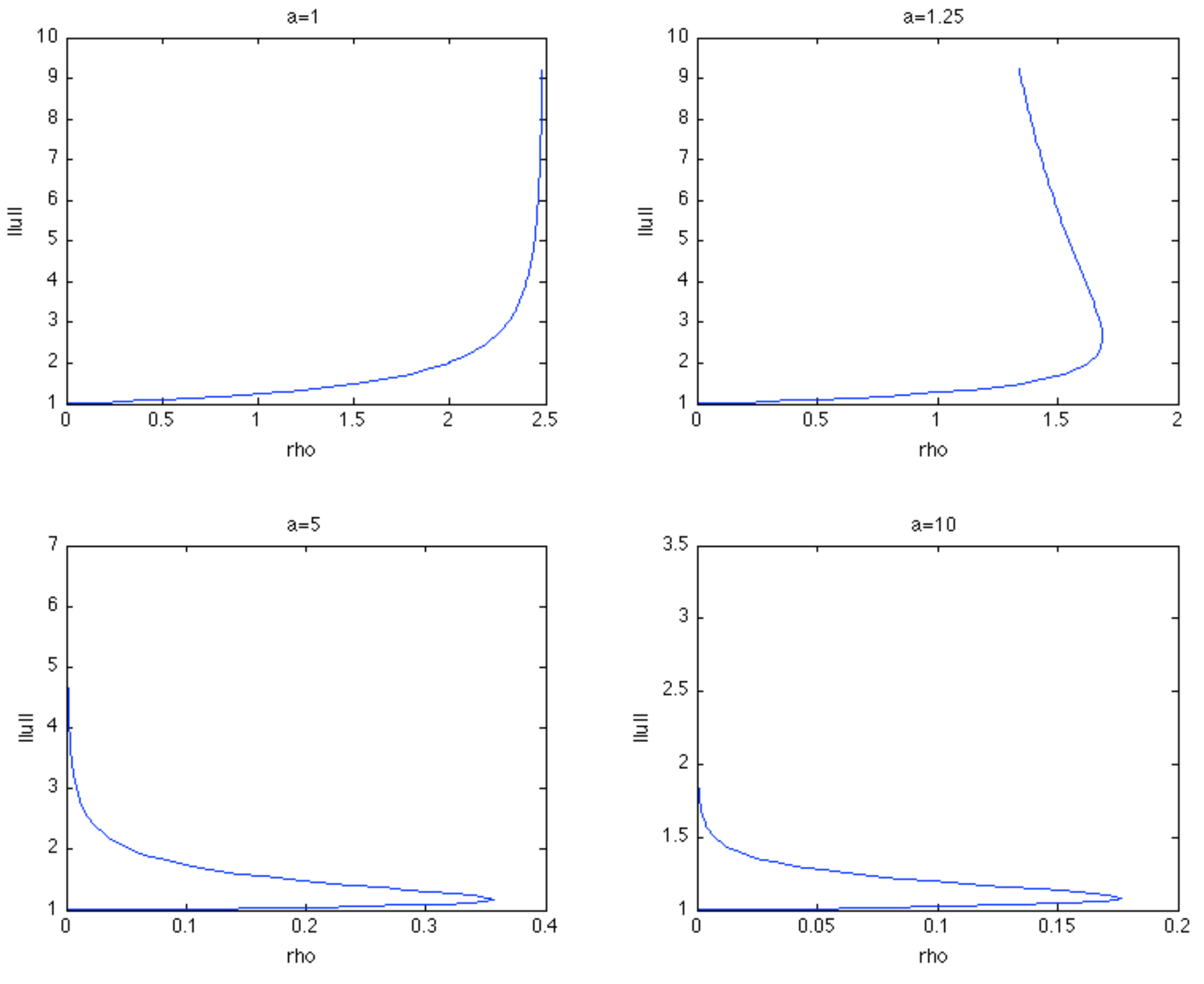}
\end{center}
\caption{\label{fig:bifall}Bifurcation diagram for $a=1,1.25,5,10$.}
\end{figure}
We'd like to note that, in contrast to the previous numerical evidence for non-uniqueness~\cite{PY}, rather than looking at the form of the solution curve and making an educated guess at the presence of a second solution curve, we confirm that there is one and \textit{only one} additional solution curve which connects to the primary solution set in a quadratic curve and there are no additional branches or any solutions past $\rho_c$.

Now we look at two representative solutions from the two branches. Taking the two solutions at $\rho=0.2$ for the original problem $a=5$, we see
\begin{figure}[tbh]
\begin{center}
\includegraphics[scale=0.68]{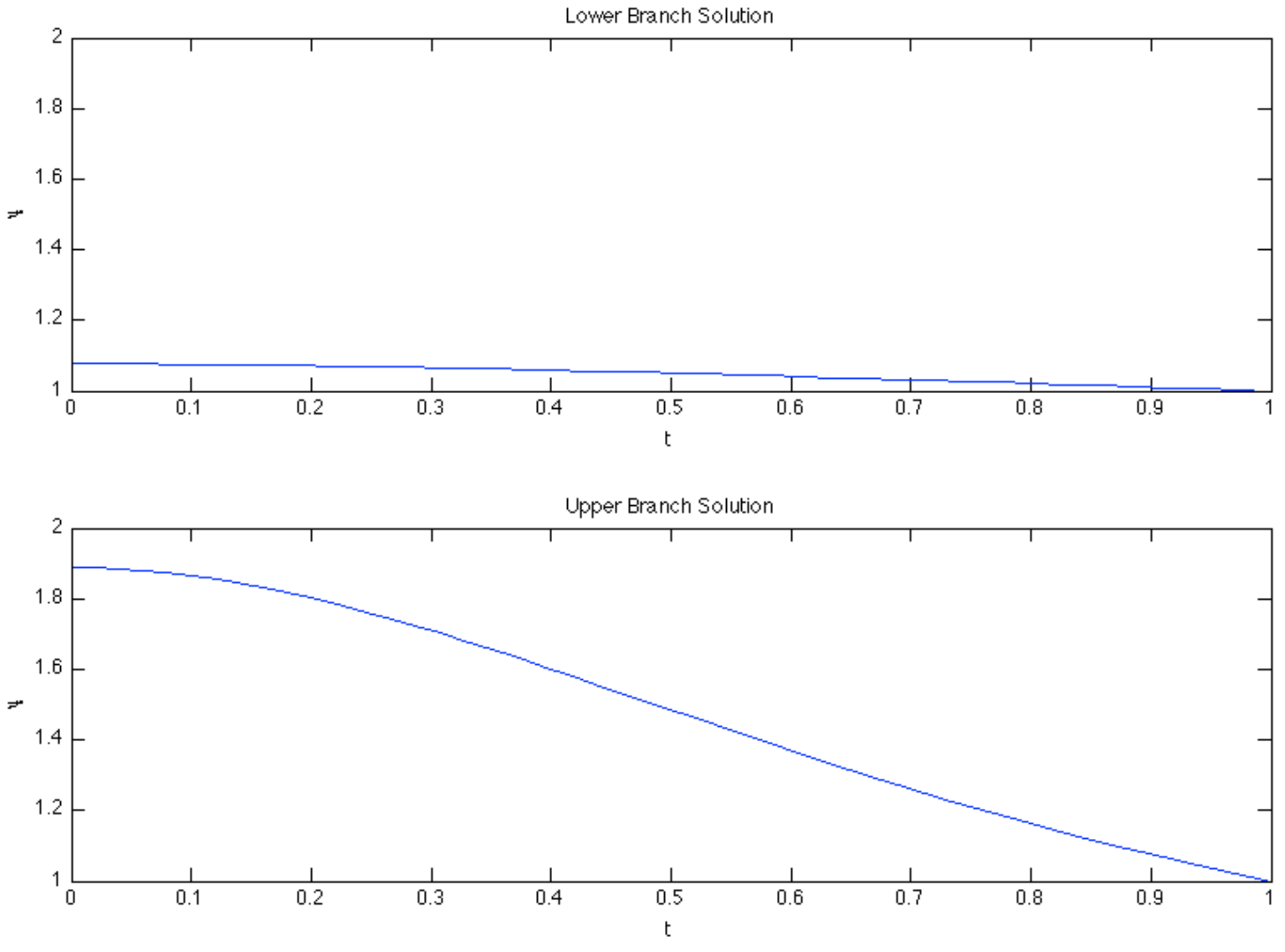}
\end{center}
\caption{Solutions for $a=5$, $\rho=0.2$.}
\end{figure}
We note that this confirms the results of ~\cite{BMP}, as the solutions have the form of a Sobolev function and furthermore the conformal factor is considerably larger, suggesting a higher ADM energy, for the upper branch solutions. 

\section{Discussion}\label{S:disc}
The Conformal Thin Sandwich approach has been a promising formulation of the Einstein constraints. The XCTS method appears to have the benefit of more physically intuitive and less computationally onerous specification of free data In contrast to its namesake predecessors, the conformal method and thin sandwich approach. In particular, you save the task of having to specify a divergence-free part of a symmetric tracefree tensor in the conformal method. As such it has been used extensively in the construction of the geometry of binary neutron stars and black holes (see the bibliography 5-12 in ~\cite{PY}). As such this paper provides a point of caution for numerical relativists.

As Baumgarte et al. ~\cite{BMP} noted, the upper branch corresponds to a higher level of ADM energy than the lower branch in this model of a constant density black hole. As such, in most applications with this set of asymptotically flat, spherically symmetric free data, the lower branch should be chosen as the correct, physically representative solution, but numerical relativists should be aware that a solver could reach either one and be able to identify the need to switch to the other branch. In addition, as $\rho \rightarrow \rho_c$ standard numerical procedures will become increasingly ill-conditioned. Pinning the solution to the expected branch is recommended as a solution strategy.

Finally, it should be noted that since this bifurcation analysis is complete, rather than simply constructive, and so demonstrates conclusively that 1) there are exactly two solutions for $\rho <\rho_c$ and 2) there are no solutions for $\rho >\rho_c$ for this choice of free data. The fact that the same choice of free data is consistent with two qualitatively disparate geometries suggests future investigation in the relation between the structure in the free and dependent set of data. In addition, it would be an interesting investigation as to why a higher density is inconsistent with time-symmetry, conformal flatness and spherical symmetry.  

\section{Conclusion}\label{S:conc}
In this article, we have examined the apparent bifurcation phenomena 
in the XTCS formulation of the Einstein constraints in a methodical way, 
using modern techniques in bifurcation theory and in numerical 
homotopy methods.
We first gave an overview of the Einstein constraints in
\S\ref{S:conformal}, and followed this in \S\ref{S:theory}
with a brief introduction to the mathematical foundations for, and in \S\ref{S:bifurtheory} the framework for
analyzing bifurcation phenomena in nonlinear operators equations.
In \S\ref{S:homotopy}, we developed the main ideas behind the construction of 
numerical homotopy, or path-following, methods in the numerical treatment 
of bifurcation phenomena, and in \S\ref{S:setup} we presented the set up of and in \S\ref{S:numerical}
we applied the continuation software package AUTO to the constraint problem.
We verified the presence of the fold with homotopy-based numerical methods,
confirming the earlier results.
Analyzing the Hamiltonian constraint for time-symmetric conformally flat 
initial data, as in~\cite{Walsh,BMP}, we demonstrated the existence and 
location of a critical point, evidence that the solution branch at the 
critical point is one-dimensional, and the form of the solution as continued 
past the critical point is a simple quadratic fold. We confirm Walsh's~\cite{Walsh} constructive Lyapunov-Schmidt analysis by showing numerically that there is indeed a solution point at which there is a one-dimensional kernel for the linearization of the differential equation operator, justifying the assumption made in his analysis, as well as confirming numerically that indeed the form of the solution continuation expansion coefficients are such that this equation exhibits a quadratic fold rather than branching or higher-order folds.

The techniques presented here can be viewed as providing a framework for
a more careful exploration of the solution theory for various
parameterizations of the constraint equations, as well as the geometric relationship between the free data in the model problem investigated.

\section{Acknowledgments}
   \label{sec:ack}

MH would like to express his appreciation to Herb Keller for introducing 
him to the continuation methodology and techniques used in the paper, 
while MH was a postdoc with Professor Keller at Caltech.
Professor Keller was very enthusiastic about eventually doing a careful 
numerical bifurcation analysis of the Einstein constraint equations.
MH would also like to thank Kip Thorne for the many conversations about
mathematical and numerical general relativity.

MH was supported in part by the NSF through Awards~0715146 and 0915220,
and by DOD/DTRA through Award HDTRA-09-1-0036.
VK was supported in part by NSF Award~0715146 and 0915220.

\vspace*{-0.2cm}
\bibliographystyle{abbrv}
\bibliography{refs,../bib/refs_hoku,../bib/papers,../bib/books,../bib/mjh}
\vspace*{0.2cm}

\end{document}